\documentclass{scrartcl}
\usepackage[utf8]{inputenc}

\usepackage{amsthm}
\makeatletter
\renewenvironment{proof}[1][\proofname] {\par\pushQED{\qed}\normalfont\topsep6\p@\@plus6\p@\relax\trivlist\item[\hskip\labelsep\bfseries#1\@addpunct{.}]\ignorespaces}{\popQED\endtrivlist\@endpefalse}
\makeatother
\usepackage{amsmath}
\usepackage{dsfont}
\theoremstyle{plain}
\newtheorem{Theorem}{Theorem}[section]
\newtheorem{Lemma}[Theorem]{Lemma}
\newtheorem{Corollary}[Theorem]{Corollary}

\theoremstyle{definition}
\newtheorem{Definition}[Theorem]{Definition}
\newtheorem{Notation}[Theorem]{Notation}
\newtheorem{Remark}[Theorem]{Remark}
\newcounter{condition}
\renewcommand{\thecondition}{C\arabic{condition}}
\newenvironment{Condition}{ \vspace{7pt} \\ \refstepcounter{condition} \textbf{Condition}~\textbf{\thecondition.}\itshape} {\vspace{7pt}}
\newcommand{\dd}[1]{\operatorname{d#1}}
\newcommand{\R}{\mathds{R}}
\newcommand{\E}{\mathds{E}}
\newcommand{\N}{\mathds{N}}
\newcommand{\1}{\mathds{1}}
\newcommand{\Co}{\mathcal{C}}

\newcommand{\Prob}{\mathds{P}}
\newcommand{\norm}[1]{\left\lVert #1 \right\rVert}
\newcommand{\abs}[1]{\left| #1 \right|}

 \title{Well-Posedness and Stability for a Class of Stochastic Delay Differential Equations with Singular Drift}
\author{ Stefan Bachmann\\  {\small Email: bachmann@math.uni-leipzig.de} \\ \multicolumn{1}{p{.55\textwidth}}{\normalsize\centering\emph{Institut f\"ur Mathematik, Universit\"at Leipzig, Augustusplatz 10, 04109 Leipzig, Germany}}}
 
\begin{document}
\maketitle
\begin{abstract}
	\textbf{Abstract}\\
	In this paper we prove well-posedness and stability of a class of stochastic delay differential equations with singular drift. Moreover, we show local well-posedness under localized assumptions. 
\end{abstract}
\section{Introduction and Main Results} 
In this paper we prove well-posedness and stability results for stochastic delay differential equations of the form
\begin{align}
	\label{deq}
	\dd{X}(t)=V(t,X_t)\dd{t}+b(t,X(t))\dd{t}+\sigma(t,X(t))\dd{W}(t)
\end{align}
where $b:\R_{\geq0}\times\R^d\to\R^d$ is a Borel-function and $\sigma:\R_{\geq0}\times\R^d\to\R^{d\times d}$ is measurable, bounded and non-degenerate, $V:\R_{\geq0}\times C\left([-r,0],\R^d\right)\to\R^d$ is measurable and
$$X_t(s):=X(t+s), \ s\in[-r,0].$$
This generalizes previous results for the non-delay case
\begin{align}
	\label{nondelayeq}
	\dd{X}(t)=b(t,X(t))\dd{t}+\sigma(t,X(t))\dd{W}(t).
\end{align}

Krylov and R\"ockner showed in \cite{KrylovRoeckner} that equation \eqref{nondelayeq} has a unique strong solution, essentially assuming $\abs{b}\in L^q_p:=L^q\left(\R_{\geq0};L^p\left(\R^d\right)\right)$, $\sigma\equiv\operatorname{Id}$ with 
\begin{align}
\label{pq}
\frac{d}{p}+\frac{2}{q}<1
\end{align}
which elaborates previous results, in particular from Zvonkin \cite{Zvonkin}, Portenko \cite{Portenko} and Veretennikov \cite{Ver}. Gy\"ongy and Mart\'inez proved existence and uniqueness theorems for non-constant $\sigma$ Lipschitz in space and $\abs{b}\in L^{2d+2}\left(\R_{\geq 0}\times\R^d\right)$ in \cite{Gyoengy2001}. Different stochastic flow theorems were studied by Gubinelli, Priola, Flandoli and Fedrizzi (cf. \cite{FlandoliGubinelli}, \cite{Fedrizzi2}).
Zhang showed existence of a unique strong solution and flow theorems for $\abs{b}\in L^q_p$ and $\sigma$ with $\abs{\nabla_x\sigma}\in L^q_{loc}\left(\R_{\geq0};L^p\left(\R^{d}\right)\right)$ in \cite{Zhang2011}. Additionally, he considered equations with Sobolev drifts and driven by $\alpha$-stable processes (cf. \cite{Zhang2}).

Our general approach is to remove the drift $b$ by Zvonkin's transformation as in \cite{Zhang2011} and to combine it with different Girsanov techniques and a stochastic Gronwall lemma from von Renesse and Scheutzow in \cite{Renesse,Scheutzow}. 

Throughout this paper, the following notation will be used
\begin{Notation}
	Fix $r>0$ and define
	$$\mathcal{C}:=C\left([-r,0],\R^d\right)$$
	equipped with the supremum norm $\norm{\cdot}_\infty$. For a process $X$ defined on $[t-r,t]$ with $t\geq0$, we define
	$$X_t(s):=X(t+s), \ s\in[-r,0].$$
	For a matrix $A\in\R^{d\times d}$, we denote by $\norm{\cdot}_{HS}$ the Hilbert-Schmidt-norm:
	$$\norm{A}_{HS}:=\sqrt{\sum_{i,j=1}^{d}\abs{A^{i,j}}^2}$$
	and for $s,t\in[-\infty,+\infty]$, we write
	\begin{align*}
		s\wedge t&:=\min(s,t),\\
		s\vee t&:=\max(s,t).
	\end{align*}
	\end{Notation}
The following conditions on $b$ and $\sigma$ are the same as in \cite{Zhang2011}.
\begin{Condition}
	\label{driftc}
$$\abs{b}\in L^q\left(\R_{\geq0};L^p\left(\R^d\right)\right)$$
for $p,q>1$ satisfying \eqref{pq}.
\end{Condition}
\begin{Condition}
	\label{sigmac}
	The diffusion coefficient $\sigma$ is uniformly continuous in $x\in\R^d$ locally uniformly with respect to $t\in\R_{\geq0}$, and $\sigma\sigma^\top$ is bounded and uniformly elliptic, i.e. there exists a $\kappa>0$ such that
	$$\kappa^{-1}I_{d\times d}\leq\sigma(t,x)\sigma(t,x)^\top\leq \kappa I_{d\times d} \ \forall x\in\R^d,t\in\R_{\geq0}.$$
\end{Condition}
\begin{Condition}
	\label{dsigmac}
	For the same $p,q\in(1,\infty)$ as in condition \eqref{driftc}, one has for the distributional gradient of $\sigma$
	$$\abs{\nabla_x\sigma^{i,j}}\in L^q_{loc}\left(\R_{\geq0};L^p\left(\R^d\right)\right), \ i,j=1,\dots,d.$$
\end{Condition}
Now, we state our conditions on the functional drift $V$:
\begin{Condition}
	\label{delayc}
	 The function $V:\R_{\geq0}\times\Co\to\R^d$ is assumed to be sublinear in the sense that there exists a monotone increasing function $g:\R_{\geq0}\to\R_{\geq0}$ with
	$$\limsup_{r\to\infty}\frac{g(r)}{r}=0$$
	such that
	$$\abs{V(t,x)}\leq g(\norm{x}_\infty)$$
	for all $t\in\R_{\geq0}$, $x\in\Co$.
\end{Condition}
\begin{Condition}
	\label{lip}
	There exists some $K>0$ such that
	$$\abs{V(t,x)-V(t,y)}\leq K\norm{x-y}_\infty$$
	for all $t\in\R_{\geq0}$, $x,y\in\Co$.
\end{Condition}
\begin{Definition}
Define for $0\leq S\leq T<\infty$ and $p,q\in(1,\infty)$
$$
\begin{array}{ll}
L^q_p(S,T):=L^q\left([S,T];L^p\left(\R^d\right)\right), &L^q_p(T):=L^q_p(0,T),\\
\mathds{H}^q_{2,p}(S,T):=L^q\left([S,T];W^{2,p}\left(\R^d\right)\right), &\mathds{H}^q_{2,p}(T):=\mathds{H}^q_{2,p}(0,T),\\
H^q_{2,p}(S,T):=W^{1,q}\left([S,T];L^p\left(\R^d\right)\right)\cap\mathds{H}^q_{2,p}(S,T),  &H^q_{2,p}(T):=H^q_{2,p}(0,T),
\end{array}
$$
equipped with the norm
$$\norm{u}_{H^q_{2,p}(S,T)}:=\norm{\partial_tu}_{L^q_p(S,T)}+\norm{u}_{\mathds{H}^q_{2,p}(S,T)}, \ u\in H^q_{2,p}(S,T).$$
\end{Definition}
\begin{Definition}
	Throughout this paper, we fix a standard $d$-dimensional Brownian motion $W=(W(t))_{t\geq0}$ defined on some filtrated probability space $(\Omega,\mathcal{F},\Prob,(\mathcal{F}_t)_{t\geq0})$ satisfying the usual conditions. Let $X$ be a local $(\mathcal{F}_t)_{t\geq0}$-adapted semimartingale which solves equation \eqref{deq} on $[0,\tau]$, for some $(\mathcal{F}_t)_{t\geq0}$-stopping time $\tau$, with respect to some initial condition $X_0=\xi$ for a $\Co$-valued, $\mathcal{F}_0$-measurable random variable $\xi$. Then we write $X\in\mathcal{S}^\tau(\xi)$.
\end{Definition}
Our main result reads as follows.
\begin{Theorem}
	\label{thm}
	Assume all conditions \eqref{driftc}, \eqref{sigmac}, \eqref{dsigmac}, \eqref{delayc} and \eqref{lip}. Then (local) pathwise uniqueness holds and there exists a global strong solution, which has almost surely $\alpha$-H\"older continuous paths on every bounded interval for any $0<\alpha<1/2$. Additionally, for any $\gamma\geq1$, $T>0$, $R>0$ and two solutions $X\in\mathcal{S}^T(x)$, $\hat{X}\in\mathcal{S}^T(\hat{x})$ where $x,\hat{x}\in\Co$ with $\max\left(\norm{x}_\infty,\norm{\hat{x}}_\infty\right)\leq R$, one has
	$$\E\norm{X_{t}-\hat{X}_{t}}_\infty^\gamma\leq C\norm{x-\hat{x}}_\infty^\gamma, \ 0\leq t\leq T$$
	with a constant  $C=C(\gamma,T,R,d,q,p,\kappa,K,\norm{b}_{L^q_p(T)},\norm{\nabla\sigma}_{L^q_p(T)},g)$.
\end{Theorem}
We can also formulate a localized version of our main result as follows.
\begin{Theorem}
	\label{locthm}
	For any $n\in\N$, let $B_n:=\{x\in\R^d:\abs{x}\leq n\}$ and assume
	\begin{enumerate}
		\item $b^i,\abs{\nabla_x\sigma^{i,j}}\in L^{q_n}_{loc}\left(\R_{\geq0};L^{p_n}\left(B_n\right)\right)$, $i,j=1,\dots,d$ where $p_n$ and $q_n$ satisfy \eqref{pq},
		\item there is a sequence of $\kappa_n>0$ such that
		$$\kappa_n^{-1}I_{d\times d}\leq\sigma(t,x)\sigma(t,x)^\top\leq \kappa_nI_{d\times d} \ \forall x\in B_n,t\in[0,n]$$
		and $\sigma$ is uniformly continuous in $x\in B_n$ uniformly with respect to $t\in[0,n]$,
		\item $V$ is locally bounded on $\R_{\geq0}\times\Co$ and for every compact $\mathcal{K}\subset\Co$ and $T>0$, there exists a constant $C_{\mathcal{K},T}$ such that
		$$\abs{V(t,x)-V(t,y)}\leq C_{\mathcal{K},T}\norm{x-y}_\infty \ \forall t\in[0,T],x,y\in\mathcal{K}.$$
	\end{enumerate}
	Then pathwise uniqueness holds and for every $x\in\Co$, there exists a maximal local solution until an (predictable) explosion time $\zeta$, i.e. $X$ solves equation \eqref{deq} on $[0,\zeta)$ with $X_0=x$ and
	$$\inf\left\{t\geq0:X_t\notin\mathcal{K}\right\}<\zeta\text{ on }\{\zeta<\infty\}$$
	for all compact $\mathcal{K}\subset\Co$.
\end{Theorem}
Finally, if $\sigma$ has no space dependence, we can relax the Lipschitz condition on $V$ as follows.
\begin{Theorem}
	\label{constthm}
	Assume all conditions from Theorem \ref{thm}. If $\sigma$ has no space dependence, i.e.
	$$\sigma(t,x)=\sigma(t), \ t\in\R_{\geq0},x\in\R^d,$$
	then one can replace the Lipschitz-condition \eqref{lip} with $$\abs{V\left(t,x+\gamma_t\right)-V(t,x)}\leq K\norm{\gamma_t}_\infty \ \forall t\geq0,\gamma\in\mathcal{H}_t,x\in\Co$$
	for some $K>0$ where
	\begin{align*}
		\mathcal{H}_t&:=\left\{\gamma\in H^1\left([-r,t],\R^d\right):\gamma(s)=0,-r\leq s\leq0\right\},\\
		H^1\left([-r,t],\R^d\right)&:=\left\{\gamma:[-r,t]\to\R^d\text{ absolutely continuous with }\int_{0}^{t}\abs{\dot{\gamma}(s)}^2\dd{s}<\infty\right\}
	\end{align*}
	to obtain pathwise uniqueness and global existence of a solution of equation \eqref{deq}.
\end{Theorem}
This kind of condition is strongly related to Malliavin-differentiable functions with bounded Malliavin-derivative, see for example \cite{enchev1993}.
\begin{Remark}
	An example for a discontinuous functional $V:\Co\to\R^d$, which fulfills the condition stated in Theorem \ref{constthm}, is the following:
	$$V(x)=
	\begin{cases}
		0,						&	x\in H^1\left([-r,0],\R^d\right),\\
		\int_{-r}^{0}x(t)\dd{t}	&	\text{otherwise}.
	\end{cases}
	$$ 
\end{Remark}
\section{Existence}
\subsection{Krylov-Type Estimates for the Non-Delay-Case}
In this subsection we consider the case $V\equiv0$.
\begin{Theorem}
	\label{Kr1}
	Assume condition \eqref{sigmac} and that $b$ is bounded and measurable. Furthermore, let $X\in\mathcal{S}^\tau(\xi)$ for some $\mathcal{F}_0$-measurable, $\Co$-valued random variable $\xi$ and stopping time $\tau$. Let $T_0>0$ and $p',q'\in(0,\infty)$ be given with
	$$\frac{d}{p'}+\frac{2}{q'}<2,$$
	there exists a constant $C(d,p',q',T_0,\kappa,\norm{b}_\infty)$ such that for all $f\in L^{q'}_{p'}(T_0)$ and $0\leq S\leq T\leq T_0$,
	$$\E\left[\int_{S\wedge\tau}^{T\wedge\tau}f(s,X(s))\dd{s}\bigg|\mathcal{F}_S\right]\leq C\norm{f}_{L^{q'}_{p'}(S,T)}.$$
\end{Theorem}
\begin{proof}
	See \cite{Zhang2011}.
\end{proof}
\begin{Lemma}
		\label{Kr3}
		Assume conditions \eqref{driftc}, \eqref{sigmac}, \eqref{dsigmac} and let $X\in\mathcal{S}^{T}(\xi)$ for some $\mathcal{F}_0$-measurable, $\Co$-valued random variable $\xi$ and $T>0$. Let $p',q'\in(0,\infty)$ be given with
		$$\frac{d}{p'}+\frac{2}{q'}<2.$$
		Then for every $R\geq0$, there exists a constant $C_R=C_R(p,q,p',q',d,T,\kappa,\norm{b}_{L^q_p(T)})$ such that
		$$\E\exp\left(\int_{0}^{T}f(t,X(t))\dd{t}\right)\leq C_R$$
		for all $f\in L^{q'}_{p'}(T)$ with $\norm{f}_{L^{q'}_{p'}(T)}\leq R$. Additionally, one has
		$$\E\int_{0}^{T}f(t,X(t))\dd{t}\leq C\norm{f}_{L^{q'}_{p'}(T)}$$
		for some constant $C>0$.
\end{Lemma}	
\begin{proof}
	Consider the global strong solution $M$ of the stochastic differential equation
	\begin{align*}
		\dd{M}(t)&=\sigma(t,M(t))\dd{W}(t),\\
		M_0&=\xi.
	\end{align*}
	Let $f\in L^{q'}_{p'}(T)$ with $\norm{f}_{L^{q'}_{p'}(T)}\leq R$. Since the inequality for $p'$ and $q'$ is strict, one can choose $\delta>1$ small enough such that
	$$\frac{d\delta}{p'}+\frac{2\delta}{q'}<2.$$
	By Theorem \ref{Kr1}, one has
	$$\E\left[\int_{S}^{T}\abs{f(t,M(t))}^\delta\dd{t}\bigg|\mathcal{F}_S\right]\leq C\norm{f}_{L^{q'}_{p'}(S,T)}^{\delta}$$
	for some constant $C(d,p',q',\kappa,\delta,T)$. Now, choose
	$$\varepsilon:=\frac{1}{2C\norm{f}_{L^{q'}_{p'}(T)}^\delta\vee1}.$$
	By Lemma \ref{Khas} and Young's inequality, one obtains
	\begin{align*}
		\E\exp\left(\int_{0}^{T}\abs{f(t,M(t))}\dd{s}\right)&\leq\E\exp\left( \int_{0}^{T}\varepsilon\abs{f(t,M(t))}^\delta\dd{s}+C_{\varepsilon,\delta,T}\right)\\
		&\leq2e^{C_{\varepsilon,\delta,T}}
	\end{align*}
	with some $C_{\varepsilon,\delta,T}>0$. Due to condition \eqref{sigmac} and
	$$\frac{d}{p}+\frac{2}{q}<1,$$
	one can apply the same method from above for $t\mapsto\sigma^{-1}(t,M(t))b(t,M(t))$ to obtain
	$$\E\exp\left(6\int_{0}^{T}\abs{\sigma^{-1}b(t,M(t))}^2\dd{t}\right)<\infty.$$
	In particular, the process $t\mapsto\sigma(t,M(t))^{-1}b(t,M(t))$ fulfills the Novikov condition and $(M,\tilde{W},\mathds{Q})$ is a weak solution of the equation
	\begin{align*}
		\dd{M}(t)&=b(t,M(t))\dd{t}+\sigma(t,M(t))\dd{\tilde{W}}(t),\\
		M_0&=\xi
	\end{align*}
	with the probability measure
	$$\dd{\mathds{Q}}:=\exp\left(\int_{0}^{T}\sigma(t,M(t))^{-1}b(t,M(t))\cdot\dd{W}(t)-\frac{1}{2}\int_{0}^{T}\abs{\sigma(t,M(t))^{-1}b(t,M(t))}^2\dd{t}\right)\dd{\Prob}$$
	and the Brownian motion
	$$\tilde{W}(t):=W(t)-\int_{0}^{t}\sigma(s,M(s))^{-1}b(s,M(s))\dd{s}.$$
	Due to Theorem 1.1 in \cite{Zhang2011}, uniqueness in distribution holds for weak solutions of the SDE
	\begin{align*}
	\dd{X}(t)&=b(t,X(t))\dd{t}+\sigma(t,X(t))\dd{W}(t),\\
	X_0&=\xi
	\end{align*}
	and the estimates from above provide
	\begin{align*}
		&\E_{\Prob}\exp\left(\int_{0}^{T}f(t,X(t))\dd{t}\right)\\
		=&\E_{\mathds{Q}}\exp\left(\int_{0}^{T}f(t,M(t))\dd{t}\right)\\
		=&\E_{\Prob}\exp\bigg(\int_{0}^{T}f(t,M(t))\dd{t}+\int_{0}^{T}\sigma(t,M(t))^{-1}b(t,M(t))\cdot\dd{W}(t)\\
		&\hspace{140pt}-\frac{1}{2}\int_{0}^{T}\abs{\sigma(t,M(t))^{-1}b(t,M(t))}^2\dd{t}\bigg)\\
		\leq&\left[\E_{\Prob}\exp\left(\int_{0}^{T}2f(t,M(t))\dd{t}\right)\right]^\frac{1}{2}\bigg[\E_\Prob\exp\bigg(2\int_{0}^{T}\sigma(t,M(t))^{-1}b(t,M(t))\cdot\dd{W}(t)\\
		&\hspace{190pt}-\int_{0}^{T}\abs{\sigma(t,M(t))^{-1}b(t,M(t))}^2\dd{t}\bigg)\bigg]^\frac{1}{2}\\
		\leq&\left[\E_{\Prob}\exp\left(\int_{0}^{T}2f(t,M(t))\dd{t}\right)\right]^\frac{1}{2}
		\cdot\left[\E_\Prob\exp\left(6\int_{0}^{T}\abs{\sigma(t,M(t))^{-1}b(t,M(t))}^2\dd{t}\right)\right]^\frac{1}{4}\\
		\leq& C_R
	\end{align*}
	for a constant $C_R=C_R(d,p,q,p',q',d,T,\kappa,\norm{b}_{L^q_p(T)}).$ By Theorem \ref{Kr1}, one has
	$$\int_{0}^{T}f(t,M(t))\dd{s}\to0\text{ in probability}$$
	if $\norm{f}_{L^{q'}_{p'}(T)}\to0$. Together with the exponential bound from above, it follows
	$$\E\int_{0}^{T}f(t,X(t))\dd{t}\to0$$
	if $\norm{f}_{L^{q'}_{p'}(T)}\to0$. Consequently, the linear operator $A:L^{q'}_{p'}(T)\to\R$ given by
	$$f\mapsto\E\int_{0}^{T}f(t,X(t))\dd{t}$$
	is continuous, which provides the existence of the desired constant.
\end{proof}
\begin{Remark}
	The previous lemma is a version of Theorem 2.2 in \cite{Zhang2011} but with relaxed assumptions on $f$. However, the proof is based on the pathwise uniqueness, which has been proven in \cite{Zhang2011}.
\end{Remark}
\subsection{Existence}
From now on, we drop the assumption $V\equiv0$.
\begin{Lemma}
	Assume condition $\eqref{sigmac}$ and consider a global weak solution $(M,W)$ for the equation
	$$\dd{M}(t)=\sigma(t,M(t))\dd{W}(t).$$
	Then for any $T>0$ and $0\leq\alpha<(2d\kappa T)^{-1}$, it holds
	$$\E\exp\left(\alpha\sup_{0\leq t\leq T}\abs{M(t)}^2\right)\leq\frac{4}{\sqrt{1-2\alpha d\kappa T}}\E\exp\left(\frac{\alpha}{1-2\alpha d\kappa T}\abs{M(0)}^2\right)-3.$$
\end{Lemma}
\begin{proof}
	Let $M$ be a weak solution of the equation above. By conditioning on $M(0)$, it is sufficient to show the estimate for constant initial values $M(0)=\xi\in\R^d$. Now, define
	$$N(t):=\int_{0}^{t}\sigma(t,M(t))\dd{W}(t).$$
	From $\eqref{sigmac}$ it follows that for the quadratic variation of each coordinate $N^i, \ i=1,\dots,d$ of $N$
	$$\langle N^i\rangle(T)\leq T\kappa.$$
	By time-transformation and a simple computation using the reflection principle for Brownian motions, it follows for any $0\leq\alpha<(2d\kappa T)^{-1}$
	\begin{align*}
		&\E\exp\left(\alpha d\sup_{0\leq t\leq T}\left(\xi^i+ N^i(t)\right)^2\right)\\
		\leq&\E\exp\left(\alpha d\sup_{0\leq t\leq \kappa T}\left(\abs{\xi^i}+\abs{W(t)}\right)^2\right)\\
		=&1+\int_{1}^{\infty}\Prob\left[\exp\left(\alpha d\left(\abs{\xi^i}+\sup_{0\leq t\leq \kappa T}\abs{W(t)}\right)^2\right)\geq x\right]\dd{x}\\
		=&1+\int_{1}^{\infty}\Prob\left(\abs{\xi^i}+\sup_{0\leq t\leq \kappa T}\abs{W(t)}\geq\sqrt{\frac{\ln(x)}{\alpha d}}\right)\dd{x}\\
		\leq&1+4\int\limits_{1}^{\infty}\Prob\left(\abs{\xi^i}+W(\kappa T)\geq\sqrt{\frac{\ln(x)}{\alpha d}}\right)\dd{x}\\
		\leq&4\E\exp\left(\alpha d\left(\abs{\xi^i}+W(\kappa T)\right)^2\right)-3\\
		=&\frac{4}{\sqrt{1-2\alpha d\kappa T}}\exp\left(\frac{\alpha d}{1-2\alpha d\kappa T}\abs{\xi^i}^2\right)-3
	\end{align*}
	where $W$ is some one-dimensional Brownian motion. Finally, by H\"older's inequality, one obtains
	\begin{align*}
		&\E\exp\left(\alpha\sup_{0\leq t\leq T}\abs{M(t)}^2\right)\\
		\leq&\E\exp\left(\alpha\sum_{i=1}^{d}\sup_{0\leq t\leq T}\left(\xi^i+N^i(t)\right)^2\right)\\
		\leq&\sqrt[d]{\prod\limits_{i=1}^d\E\exp\left(\alpha d\sup_{0\leq t\leq T}\left(\xi^i+N^i(t)\right)^2\right)}\\
		\leq&\frac{4}{\sqrt{1-2\alpha d\kappa T}}\exp\left(\frac{\alpha}{1-2\alpha d\kappa T}\abs{\xi}^2\right)-3,
	\end{align*}
	where the following inequality was used
	$$\sqrt[d]{x_1\cdots x_d}+3\leq\sqrt[d]{(x_1+3)\cdots(x_d+3)} \ \forall x_1,\dots,x_d\geq0.$$
\end{proof}
\begin{Corollary}
	\label{exp1}
	Assume conditions \eqref{driftc}, \eqref{sigmac}, \eqref{dsigmac} and consider a solution of \eqref{deq} with $V\equiv0$. Then one has for any $T>0$ and $0\leq\alpha<(4d\kappa T)^{-1}$ the inequality
	$$\E\exp\left(\alpha\sup\limits_{-r\leq t\leq T}\abs{X(t)}^2\right)\leq\frac{C}{\sqrt[4]{1-4\alpha d\kappa T}}\E\exp\left(\frac{\alpha}{1-4\alpha d\kappa T}\norm{X_0}_\infty^2\right)$$
	for a constant $C=C(d,p,q,\kappa,\norm{b}_{L^q_p(T)})$.
\end{Corollary}
\begin{proof}
	By conditioning on $X_0$, it is sufficient to show the estimate for constant initial values $X_0=\xi\in\Co$. By condition \eqref{sigmac} and Lemma \ref{Kr3}, the Novikov condition
	$$\E\exp\left(\frac{1}{2}\int_{0}^{T}\abs{\sigma(t,X(t))^{-1}b(t,X(t))}^2\dd{t}\right)<\infty$$
	is fulfilled. Additionally, the bound is independent of the initial value $\xi$. Thus, under the probability measure
	\begin{align*}
		&\dd{\mathds{Q}}\\
		:=&\exp\left(-\int_{0}^{T}\sigma(t,X(t))^{-1}b(t,X(t))\cdot\dd{W}(t)-\frac{1}{2}\int_{0}^{T}\abs{\sigma(t,X(t))^{-1}b(t,X(t))}\dd{t}\right)\dd{\Prob},
	\end{align*}
	the process
	$$\tilde{W}(t):=W(t)+\int_{0}^{t}\sigma(s,X(s))^{-1}b(s,X(s))\dd{s}$$
	is a Brownian motion and $X$ solves the equation
		$$\dd{X}(t)=\sigma(t,X(t))\dd{\tilde{W}}(t)$$
	on $[0,T]$.	
	
	It follows
	\begin{align*}
		&\E_\Prob\exp\left(\alpha\sup\limits_{-r\leq t\leq T}\abs{X(t)}^2\right)\\
		=&\E_{\mathds{Q}}\bigg[\exp\left(\alpha\sup\limits_{-r\leq t\leq T}\abs{X(t)}^2\right)\cdot\exp\bigg(\int_{0}^{T}\sigma(t,X(t))^{-1}b(t,X(t))\cdot\dd{W}(t)\\
		&\hspace{6cm}+\frac{1}{2}\int_{0}^{T}\abs{\sigma(t,X(t))^{-1}b(t,X(t))}^2\dd{t}\bigg)\bigg]\\
		=&\E_{\mathds{Q}}\bigg[\exp\left(\alpha\sup\limits_{-r\leq t\leq T}\abs{X(t)}^2\right)\cdot\exp\bigg(\int_{0}^{T}\sigma(t,X(t))^{-1}b(t,X(t))\cdot\dd{\tilde{W}}(t)\\
		&\hspace{6cm}-\frac{1}{2}\int_{0}^{T}\abs{\sigma(t,X(t))^{-1}b(t,X(t))}^2\dd{t}\bigg)\bigg]\\
		\leq&\left[\E_{\mathds{Q}}\exp\left(2\alpha\sup\limits_{-r\leq t\leq 	T}\abs{X(t)}^2\right)\right]^{\frac{1}{2}}
		\cdot\bigg[\E_{\mathds{Q}}\exp\bigg(2\int_{0}^{T}\sigma(t,X(t))^{-1}b(t,X(t))\cdot\dd{\tilde{W}}(t)\\
		&\hspace{180pt}-\int_{0}^{T}\abs{\sigma(t,X(t))^{-1}b(t,X(t))}^2\dd{t}\bigg)\bigg]^\frac{1}{2}\\
		\leq&\left[\E_{\mathds{Q}}\exp\left(2\alpha\sup\limits_{-r\leq t\leq T}\abs{X(t)}^2\right)\right]^\frac{1}{2}\cdot\left[\E_{\mathds{Q}}\exp\left(6\int_{0}^{T}\abs{\sigma(t,X(t))^{-1}b(t,X(t))}^2\dd{t}\right)\right]^{\frac{1}{4}}\\
		\leq&\frac{C}{\sqrt[4]{1-4\alpha d\kappa T}}\exp\left(\frac{\alpha}{1-4\alpha d\kappa T}\norm{\xi}_\infty^2\right)	
	\end{align*}
	for a constant $C=C(d,p,q,\kappa,\norm{b}_{L^q_p(T)})$ because of condition \eqref{sigmac}, Lemmas \ref{Kr3} and \ref{exp1}.
\end{proof}
\begin{Theorem}
	\label{Existence}
	Assume conditions \eqref{driftc}, \eqref{sigmac}, \eqref{dsigmac} and \eqref{delayc}. Then for any initial distribution $\mu\in\mathcal{P}(\Co)$, there exists a global weak solution $(X,W)$ of \eqref{deq} with $X_0\sim\mu$.
\end{Theorem}
\begin{proof}
	Due to Theorem 1.1 in \cite{Zhang2011}, there exists a global strong solution $X$ of the equation
	\begin{align*}
		\dd{X}(t)&=b(t,X(t))\dd{t}+\sigma(t,X(t))\dd{W}(t),\\
		X_0&\sim\mu
	\end{align*}
	on the filtrated probability space $\left(\Omega,\mathcal{F},\Prob,\left(\mathcal{F}_t\right)_{t\geq0}\right)$ which is considered in this paper. Without loss of generality, one can assume that $\mathcal{F}$ is generated by the filtration $(\mathcal{F}_t)_{t\geq0}$ which is the augmented filtration generated by $X_0$ and $W$.\\
	Firstly, assume that $\mu$ has bounded support. For $T>0$, by Corollary \ref{exp1},
	\begin{align*}
	&\dd{\mathds{Q}_T}\\
	:=&\exp\left(\int_{0}^{T}\sigma(t,X(t))^{-1}V(t,X_t)\cdot\dd{W}(t)-\frac{1}{2}\int_{0}^{T}\abs{\sigma(t,X(t))^{-1}V(t,X_t)}^2\dd{t}\right)\dd{\Prob}
	\end{align*}
	is a probability measure. Under $\mathds{Q}_T$, the process
	$$\tilde{W}:=W(t)-\int_{0}^{t}\sigma(s,X(s))^{-1}V(s,X_s)\dd{s}, \ t\geq0$$
	is a Brownian motion on $[0,T]$ and $X$ is a weak solution of
	\begin{align*}
		\dd{X}(t)&=V(t,X_t)\dd{t}+b(t,X(t))\dd{t}+\sigma(t,X(t))\dd{\tilde{W}}(t)
	\end{align*}
	on $[0,T]$. Additionally, it holds for any $0<T_1<T_2$
	$$\mathds{Q}_{T_1}(A)=\mathds{Q}_{T_2}(A) \ \forall A\in\mathcal{F}_{T_1}.$$
	Now, let $\mathds{Q}$ be the probability measure uniquely defined by
	$$\mathds{Q}(A):=\mathds{Q}_T(A), \ T>0,A\in\mathcal{F}_T.$$
	Then $(X,\tilde{W},\mathds{Q})$ is a global weak solution. In the general case, choose a sequence of bounded, disjoint, measurable subsets $(A_n)_{n\in I}\subset\Co$ with
	$$\mu(A_n)>0 \ \forall n\in I$$
	where $I$ is some countable index set. Now, define the probability measures
	$$\mu^n(\cdot):=\mu\left(\cdot\big|A_n\right).$$
	By the discussion from above, there exists for each $n\in I$ a probability measure $\Prob^n$ and a Brownian motion $W^n$ such that $(X,W^n,\Prob^n)$ is a global weak solution with initial distribution $\mu^n$. Now, define the probability measure
	$$\hat{\Prob}(A):=\sum\limits_{n\in I}\Prob^n(A)\mu(A_n), \ A\in\mathcal{F}$$
	and the process
	$$\hat{W}(\omega,t):=W^n(\omega,t)\text{ if }X_0(\omega)\in A_n.$$
	Now, let $f:\Omega\to\R$ be measurable with finite moment with respect to $\hat{\Prob}$, $s\in\R_{\geq0}$ and $A\in\mathcal{F}_s$. Then one has
	\begin{align*}
		&\E_{\hat{\Prob}}\left(\1_Af\right)\\
		=&\sum_{n\in I}\mu(A_n)\E_{\Prob^n}(\1_Af)\\
		=&\sum_{n\in I}\mu(A_n)\E_{\Prob^n}\left(\1_A\E_{\Prob^n}\left(f|\mathcal{F}_s\right)\right)\\
		=&\sum_{n\in I}\mu(A_n)\E_{\Prob^n}\left(\1_{X_0\in A_n}\1_A\E_{\Prob^n}\left(f|\mathcal{F}_s\right)\right)\\
		=&\sum_{n\in I}\E_{\hat{\Prob}}\left(\1_{X_0\in A_n}\1_A\E_{\Prob^n}\left(f|\mathcal{F}_s\right)\right).
	\end{align*}
	It follows
	$$\E_{\hat{\Prob}}\left(f|\mathcal{F}_s\right)=\sum_{n\in I}\1_{X_0\in A_n}\E_{\Prob^n}\left(f|\mathcal{F}_s\right).$$
	Since each $W^n$ is a Brownian motion and a martingale under $\Prob^n$, one has for $0\leq s\leq t$
	\begin{align*}
		\E_{\hat{\Prob}}\left(\hat{W}(t)\Big|\mathcal{F}_s\right)&=\sum_{n\in I}\1_{X_0\in A_n}\E_{\Prob^n}\left(\hat{W}(t)\Big|\mathcal{F}_s\right)\\
		&=\sum_{n,m\in I}\1_{X_0\in A_n}\E_{\Prob^n}\left(\1_{X_0\in A_m}W^m(t)\big|\mathcal{F}_s\right)\\
		&=\sum_{n,m\in I}\1_{X_0\in A_n}\1_{X_0\in A_m}\E_{\Prob^n}\left(W^m(t)\big|\mathcal{F}_s\right)\\
		&=\sum_{n\in I}\1_{X_0\in A_n}\E_{\Prob^n}\left(W^n(t)\big|\mathcal{F}_s\right)\\
		&=\sum_{n\in I}\1_{X_0\in A_n}W^n(s)\\
		&=\hat{W}(s)
	\end{align*}
	and analogously
	$$\E\left(\hat{W}^i(t)\hat{W}^j(t)\Big|\mathcal{F}_s\right)=\delta_{ij}(t-s)$$
	where $i,j=1,\dots d$. Additionally, the process $\hat{W}$ is almost surely continuous and by Levy's characterization, $\hat{W}$ is a Brownian motion on $\left(\Omega,\mathcal{F},\hat{\Prob},\left(\mathcal{F}_t\right)_{t\geq0}\right)$. Hence, $(X,\hat{W},\hat{\Prob})$ is a weak solution with initial distribution $X_0\sim\mu$.
\end{proof}
\subsection{Exponential- and Krylov-Type Estimates for the General Case}
In this section we show similar estimates like above for solutions with delay drift.
\begin{Lemma}
	\label{exp2}
	Assume conditions \eqref{driftc}, \eqref{sigmac}, \eqref{dsigmac}, \eqref{delayc} and consider a local solution $X\in\mathcal{S}^\tau(\xi)$ for some $\mathcal{F}_0$-measurable, $\Co$-valued random variable $\xi$ and a stopping time $\tau$. Then one has for any $T>0$ and $0\leq\alpha<(8d\kappa T)^{-1}$ the inequality
	$$\left[\E\exp\left(\alpha\sup\limits_{-r\leq t\leq T\wedge\tau}\abs{X(t)}^2\right)\right]^2\leq \frac{C}{\sqrt[4]{1-8\alpha d\kappa T}}\E\exp\left(\frac{2\alpha}{1-8\alpha d\kappa T}\norm{\xi}_\infty^2\right)$$
	for some constant $C=C(d,p,q,T,\kappa,\norm{b}_{L^q_p(T)},g,\xi)$.
\end{Lemma}
\begin{proof}
	Introduce the stopping times
	$$\tau_n:=\inf\left\{t\geq0:\abs{V(t,X_t)}\geq n\right\}\wedge\tau\wedge T.$$
	By the monotone convergence theorem, it suffices to show the inequality for the stopped processes $X^{\tau_n}$ with a uniform bound. The following technique is similar to the one used in \cite[p. 286-297]{liptser2001statistics}. For every $n\in\N$, there exists a strong solution of the SDE
	\begin{align*}
		\dd{Y^n}(t)&=b\left(t,Y^n(t)\right)\dd{t}+\sigma\left(t,Y^n(t)\right)\dd{W}(t), \ t\geq\tau_n,\\
		Y(\tau_n)&=X(\tau_n).
	\end{align*}
	Now, define processes $X^n$ by
	$$
	X^n(t):=
	\begin{cases}
		X(t), & t\leq\tau_n,\\
		Y^n(t), & t>\tau_n.
	\end{cases}
	$$
	Under the probability measure
	\begin{align*}
		&\dd{\mathds{Q}^n}\\
		:=&\exp\left(-\int_{0}^{\tau_n}\sigma(t,X^n(t))^{-1}V(t,X^n_t)\cdot\dd{W}(t)-\frac{1}{2}\int_{0}^{\tau_n}\abs{\sigma(t,X^n(t))^{-1}V(t,X^n_t)}^2\dd{t}\right)\dd{\Prob},
	\end{align*}
	the process
	$$W^n(t):=W(t)+\int_{0}^{t\wedge\tau_n}\sigma(s,X^n(s))^{-1}V(s,X^n_s)\dd{s}$$
	is a Brownian motion and $X^n$ is the unique strong solution of the equation
	\begin{align*}
		\dd{X^n}(t)&=b(t,X^n(t))\dd{t}+\sigma(t,X^n(t))\dd{W^n}(t),\\
		X^n_0&=X_0.
	\end{align*}
	Accordingly,
	\begin{align*}
		&\left[\E_\Prob\exp\left(\alpha\sup\limits_{-r\leq t\leq T}\abs{X^n(t)}^2\right)\right]^2\\
		=&\bigg[\E_{\mathds{Q}^n}\bigg(\exp\left\{\alpha\sup\limits_{-r\leq t\leq T}\abs{X^n(t)}^2\right\}\cdot\exp\bigg\{\int_{0}^{\tau_n}\sigma(t,X^n(t))^{-1}V(t,X^n_t)\cdot\dd{W}(t)\\
		&\hspace{170pt}+\frac{1}{2}\int_{0}^{\tau_n}\abs{\sigma(t,X^n(t))^{-1}V(t,X^n_t)}^2\dd{t}\bigg\}\bigg)\bigg]^2\\
		=&\bigg[\E_{\mathds{Q}^n}\bigg(\exp\left\{\alpha\sup\limits_{-r\leq t\leq T}\abs{X^n(t)}^2\right\}\cdot\exp\bigg\{\int_{0}^{\tau_n}\sigma(t,X^n(t))^{-1}V(t,X^n_t)\cdot\dd{W^n}(t)\\
		&\hspace{170pt}-\frac{1}{2}\int_{0}^{\tau_n}\abs{\sigma(t,X^n(t))^{-1}V(t,X^n_t)}^2\dd{t}\bigg\}\bigg)\bigg]^2\\
		\leq&\E_{\mathds{Q}^n}\exp\left(2\alpha\sup\limits_{-r\leq t\leq T}\abs{X^n(t)}^2\right)
		\cdot\E_{\mathds{Q}^n}\exp\bigg(2\int_{0}^{\tau_n}\sigma(t,X^n(t))^{-1}V(t,X^n_t)\cdot\dd{W^n}(t)\\
		&\hspace{190pt}-\int_{0}^{\tau_n}\abs{\sigma(t,X^n(t))^{-1}V(t,X^n_t)}^2\dd{t}\bigg)\\
		\leq&\E_{\mathds{Q}^n}\exp\left(2\alpha\sup\limits_{-r\leq t\leq T}\abs{X^n(t)}^2\right)\cdot\left[\E_{\mathds{Q}^n}\exp\left(6\int_{0}^{T}\abs{\sigma(t,X^n(t))^{-1}V(t,X^n_t)}^2\dd{t}\right)\right]^{\frac{1}{2}}\\
		\leq&\frac{C}{\sqrt[4]{1-8\alpha d\kappa T}}\E\exp\left(\frac{2\alpha}{1-8\alpha d\kappa T}\norm{\xi}_\infty^2\right)
	\end{align*}
	for  a constant $C=C(d,p,q,T,\kappa,\norm{b}_{L^q_p(T)},g,\xi)$ due to condition \eqref{delayc} and Corollary \ref{exp1}.
\end{proof}
\begin{Lemma}
	\label{Kr4}
	Assume conditions \eqref{driftc}, \eqref{sigmac}, \eqref{dsigmac}, \eqref{delayc} and let $X\in\mathcal{S}^\tau(\xi)$ for some $\mathcal{F}_0$-measurable, $\Co$-valued random variable $\xi$ and a stopping time $\tau$ such that
	$$\E\exp\left(\varepsilon\norm{\xi}_\infty^2\right)<\infty$$
	for some $\varepsilon>0$. Let $T>0$ and $p',q'\in(0,\infty)$ be given with
	\begin{align*}
	\frac{d}{p'}+\frac{2}{q'}<2.
	\end{align*}
	Then there exists for every $R\geq0$ a constant $C_R=C_R(\xi,p,q,p',q',d,T,\kappa,\norm{b}_{L^q_p(T)},g)$ such that
	$$\E\exp\left(\int_{0}^{T\wedge\tau}f(s,X(s))\dd{s}\right)\leq C_R$$
	for all $f\in L^{q'}_{p'}(T)$ and $\norm{f}_{L^{q'}_{p'}(T)}\leq R$. Additionally,  one has
	$$\E\int_{0}^{T\wedge\tau}f(s,X(s))\dd{s}\leq C\norm{f}_{L^{q'}_{p'}(T)}$$
	with a constant $C>0$.
\end{Lemma}	
\begin{proof}
	By condition \eqref{sigmac}, Lemma \ref{exp2} and the assumption for the initial distribution, the Novikov condition
	$$\E\exp\left(\frac{1}{2}\int_{0}^{T\wedge\tau}\abs{\sigma(t,X(t))^{-1}V(t,X_t)}^2\dd{t}\right)<\infty$$
	is fulfilled. Therefore, under the probability measure
	\begin{align*}
		&\dd{\mathds{Q}}\\
		:=&\exp\left(-\int_{0}^{T\wedge\tau}\sigma(t,X(t))^{-1}V(t,X_t)\cdot\dd{W}(t)-\frac{1}{2}\int_{0}^{T\wedge\tau}\abs{\sigma(t,X(t))^{-1}V(t,X_t)}^2\dd{t}\right)\dd{\Prob},
	\end{align*}
	the process
	$$\tilde{W}(t):=W(t)+\int_{0}^{t\wedge\tau}\sigma(s,X^n(s))^{-1}V(s,X^n_s)\dd{s}$$
	is a Brownian motion and $X$ solves the equation
	\begin{align*}
		\dd{X}(t)&=b(t,X(t))\dd{t}+\sigma(t,X(t))\dd{\tilde{W}}(t),
	\end{align*}
	on $[0,T\wedge\tau]$.
	
	It follows
	\begin{align*}
		&\E_\Prob\exp\left(\int_{0}^{T\wedge\tau}f(t,X(t))\dd{t}\right)\\
		=&\E_{\mathds{Q}}\bigg[\exp\left(\int_{0}^{T\wedge\tau}f(t,X(t))\dd{t}\right)\cdot\exp\bigg(\int_{0}^{T\wedge\tau}\sigma(t,X(t))^{-1}V(t,X_t)\cdot\dd{W}(t)\\
		&\hspace{6cm}+\frac{1}{2}\int_{0}^{T\wedge\tau}\abs{\sigma(t,X(t))^{-1}V(t,X_t)}^2\dd{t}\bigg)\bigg]\\
		=&\E_{\mathds{Q}}\bigg[\exp\left(\int_{0}^{T\wedge\tau}f(t,X(t))\dd{t}\right)\cdot\exp\bigg(\int_{0}^{T\wedge\tau}\sigma(t,X(t))^{-1}V(t,X_t)\cdot\dd{\tilde{W}}(t)\\
		&\hspace{6cm}-\frac{1}{2}\int_{0}^{T\wedge\tau}\abs{\sigma(t,X(t))^{-1}V(t,X_t)}^2\dd{t}\bigg)\bigg]\\
		\leq&\left[\E_{\mathds{Q}}\exp\left(2\int_{0}^{T\wedge\tau}f(t,X(t))\dd{t}\right)\right]^\frac{1}{2}
		\cdot\bigg[\E_{\mathds{Q}}\exp\bigg(2\int_{0}^{T\wedge\tau}\sigma(t,X(t))^{-1}V(t,X_t)\cdot\dd{\tilde{W}}(t)\\
		&\hspace{200pt}-\int_{0}^{T\wedge\tau}\abs{\sigma(t,X(t))^{-1}V(t,X_t)}^2\dd{t}\bigg)\bigg]^\frac{1}{2}\\
		\leq&\left[\E_{\mathds{Q}}\exp\left(2\int_{0}^{T\wedge\tau}f(t,X(t))\dd{t}\right)\right]^\frac{1}{2}\cdot\left[\E_{\mathds{Q}}\exp\left(6\int_{0}^{T\wedge\tau}\abs{\sigma(t,X(t))^{-1}V(t,X_t)}^2\dd{t}\right)\right]^\frac{1}{4}\\
		\leq&C_R
	\end{align*}
	where $C_R=C_R(\xi,p,q,p',q',d,T,\kappa,\norm{b}_{L^q_p(T)},g)\geq0$ because of Lemma \ref{Kr3}, condition \eqref{delayc} and Lemma \ref{exp2}. For the last statement, one can use
	$$\int_{0}^{T\wedge\tau}f(s,X(s))\dd{s}\to0\text{ in probability with respect to }\mathds{Q}$$
	if $\norm{f}_{L^{q'}_{p'}(T)}\to0$, and the bound from above, to conclude
	$$\E_{\Prob}\int_{0}^{T\wedge\tau}f(s,X(s))\dd{s}\to0$$
	if $\norm{f}_{L^{q'}_{p'}(T)}\to0$. Consequently, the linear operator $A:L^{q'}_{p'}(T)\to\R$ given by
	$$f\mapsto\E\int_{0}^{T\wedge\tau}f(s,X(s))\dd{s}$$
	is continuous, which provides the existence of the desired constant.
\end{proof}
\begin{Lemma}
	\label{Hoelder}
	Assume conditions \eqref{driftc}, \eqref{sigmac}, \eqref{dsigmac}, \eqref{delayc} and let $X\in\mathcal{S}^\tau(\xi)$ be a weak solution for some $\mathcal{F}_0$-measurable, $\Co$-valued random variable $\xi$ and a stopping time $\tau$. Then $X$ has almost surely $\alpha$-H\"older continuous paths on $[0,T\wedge\tau]$ for any $0<\alpha<1/2$ and $T>0$.
\end{Lemma}
\begin{proof}
	Let $0<\alpha<1/2$ and $T>0.$
	\begin{enumerate}
		\item $t\mapsto\int_{0}^{t\wedge\tau}V(s,X_s)\dd{s}$ has Lipschitz continuous paths on $[0,T]$ since $V$ is locally bounded.
		\item $t\mapsto\int_{0}^{t\wedge\tau}b(t,X(s))\dd{s}$ has almost surely $\alpha$-H\"older continuous paths on $[0,T]$ because $\E\int_{0}^{T\wedge\tau}\abs{b(t,X(t))}^2\dd{t}<\infty$ by Lemma \ref{Kr4}.
		\item $t\mapsto\int_{0}^{t\wedge\tau}\sigma(s,X(s))\dd{W}(s)$ has almost surely $\alpha$-H\"older continuous paths on $[0,T]$ since $\sigma$ is bounded.
	\end{enumerate}
\end{proof}
\section{Pathwise Uniqueness}
\begin{Theorem}
	\label{Uniqueness}
	Assume conditions \eqref{driftc}, \eqref{sigmac}, \eqref{dsigmac}, \eqref{delayc}, \eqref{lip}, let $\tau$ be a stopping time and $R>0$. For every two local solutions $X\in\mathcal{S}^\tau(x)$ and $\hat{X}\in\mathcal{S}^\tau(\hat{x})$ where $x,\hat{x}\in\Co$ with $\max(\norm{x}_\infty,\norm{\hat{x}}_\infty)\leq R$, every $\gamma\geq1$ and $T_0>0$, one has
	$$\E\norm{X_{t\wedge\tau}-\hat{X}_{t\wedge\tau}}_\infty^\gamma\leq C\norm{x-\hat{x}}_\infty^\gamma, \ 0\leq t\leq T_0$$
	for some constant $C$ depending only on $\gamma,d,p,q,T_0,\kappa,K,\norm{b}_{L^q_p(T_0)},\norm{\nabla\sigma}_{L^q_p(T_0)},g$ and $R$.
\end{Theorem}
By Theorem \ref{PDE}, for every $0<T\leq T_0$, there exists a solution
$$\tilde{u}(\cdot;T)\in\left(H^q_{2,p}(T_0)\right)^d$$
of the coordinatewise PDE system
\begin{align*}
	\partial_t\tilde{u}(t,x;T)+L_t\tilde{u}(t,x;T)+b(t,x)&=0,\\
	\tilde{u}(T,x;T)&=0
\end{align*}
for all $t\in[0,T]$ and $x\in\R^d$ where
$$L_tv(t,x):=\frac{1}{2}\sum_{i,j,k=1}^{d}\sigma^{i,k}(t,x)\sigma^{j,k}(t,x)\partial_i\partial_jv(t,x)+b(t,x)\cdot\nabla v(t,x), \ v\in H^q_{2,p}(T_0).$$
Additionally, it holds
$$\sup\limits_{T\in[0,T_0]}\left(\norm{\partial_t\tilde{u}^i(\cdot;T)}_{L_p^q(T)}+\norm{\tilde{u}^i(\cdot;T)}_{H^q_{2,p}(T)}\right)<\infty, \ i=1,\dots,d$$
and by the embedding Theorem \ref{embedding}, there exists a uniform $\delta$ such that for all $0\leq S\leq T$ with $T-S\leq\delta$
$$\abs{\tilde{u}(t,x;T)-\tilde{u}(t,y;T)}\leq\frac{1}{2}\abs{x-y}$$
for all $t\in[S,T]$ and $x,y\in\R^d$. Furthermore, the function
$$u(t,x;T):=\tilde{u}(t,x;T)+x$$
satisfies coordinatewise the equation
\begin{align*}
\partial_tu(t,x;T)+L_tu(t,x;T)&=0,\\
u(T,x;T)&=x.
\end{align*}
\begin{proof}
	Let $X\in\mathcal{S}^\tau(x)$ and $\hat{X}\in\mathcal{S}^\tau(\hat{x})$ for some $(\mathcal{F}_t)_{t\geq0}$-stopping time $\tau$ where $x,\hat{x}\in\Co$. Choose $T_0>0$, $\gamma\geq1$ arbitrarily and $\delta>0$ like above. By induction, it suffices to prove for every $0\leq S\leq T\leq T_0$ with $T-S\leq\delta$ the implication
	\begin{align*}
		&\E\norm{X_{S\wedge\tau}-\hat{X}_{S\wedge\tau}}_\infty^\gamma\leq C_1\norm{x-\hat{x}}_\infty^\gamma\\
		\implies&\E\norm{X_{T\wedge\tau}-\hat{X}_{T\wedge\tau}}_\infty^\gamma\leq C_2\norm{x-\hat{x}}_\infty^\gamma
	\end{align*}
	for constants $C_1$ and $C_2$ depending only on $\gamma,d,p,q,\kappa,K,T_0,\norm{b}_{L^q_p(T_0)},\norm{\nabla\sigma}_{L^q_p(T_0)},g$ and $R$. For the sake of simplicity, we write $u(\cdot):=u(\cdot;T)$. Furthermore, define
	\begin{align*}
		Y(t)&:=u(t,X(t)), \ S\wedge\tau\leq t\leq T\wedge\tau\\
		\hat{Y}(t)&:=u(t,\hat{X}(t)), \ S\wedge\tau\leq t\leq T\wedge\tau.
	\end{align*}
	By the choice of $\delta$, one has for the difference processes $Z(t):=X(t)-\hat{X}(t)$ and $\tilde{Z}(t):=Y(t)-\hat{Y}(t)$
	$$\frac{1}{2}\abs{\tilde{Z}(t)}\leq\abs{Z(t)}\leq\frac{3}{2}\abs{\tilde{Z}(t)}, \ S\wedge\tau\leq t\leq T\wedge\tau.$$
	Due to Lemma \ref{Kr4}, Lemma \ref{Ito} is applicable, which gives
	\begin{align*}
		\tilde{Z}(t)=&\int_{S\wedge\tau}^{t\wedge\tau}\left(Du(s,X(s))V(s,X_s)-D u(s,\hat{X}(s))V(s,\hat{X}_s)\right)\dd{s}\\
		&+\int_{S\wedge\tau}^{t\wedge\tau}\left(Du(s,X(s))\sigma(s,X(s))-D u(s,\hat{X}(s))\sigma(s,\hat{X}(s))\right)\dd{W}(s)
	\end{align*}
	and consequently
	\begin{align*}
		&\dd{\abs{\tilde{Z}}^{2\gamma}}(t)\\
		=&2\gamma\abs{\tilde{Z}(t)}^{2\gamma-2}\tilde{Z}(t)^\top\left(Du(t,X(t))V(t,X_t)-Du(t,\hat{X}(t))V(t,\hat{X}_t)\right)\dd{t}\\
		+&2\gamma\abs{\tilde{Z}(t)}^{2\gamma-2}\tilde{Z}(t)^\top\left(Du(t,X(t))\sigma(t,X(t))-Du(t,\hat{X}(t))\sigma(t,\hat{X}(t))\right)\dd{W}(t)\\
		+&\gamma\abs{\tilde{Z}(t)}^{2\gamma-2}\norm{Du(t,X(t))\sigma(t,X(t))-Du(t,\hat{X}(t))\sigma(t,\hat{X}(t))}_{HS}^2\dd{t}\\
		+&2\gamma(\gamma-1)\abs{\tilde{Z}(t)}^{2\gamma-4}\abs{\left(Du(t,X(t))\sigma(t,X(t))-Du(t,\hat{X}(t))\sigma(t,\hat{X}(t))\right)^\top\tilde{Z}(t)}^2\dd{t}.
	\end{align*}
	Using the boundedness of $Du$ and condition \eqref{lip} gives for $S\leq t_1\leq t_2\leq T$
	\begin{align*}
		&\abs{Z(t_2\wedge\tau)}^{2\gamma}-\abs{Z(t_1\wedge\tau)}^{2\gamma}\\
		\leq&c\int_{t_1\wedge\tau}^{t_2\wedge\tau}\norm{Z_s}_\infty^{2\gamma}\dd{s}\\
		+&c\int_{t_1\wedge\tau}^{t_2\wedge\tau}\abs{Z(s)}^{2\gamma-1}\norm{Du(s,X(s))-Du(s,\hat{X}(s))}_{HS}\abs{V(s,X_s)}\dd{s}\\
		+&c\int_{t_1\wedge\tau}^{t_2\wedge\tau}\abs{\tilde{Z}(s)}^{2\gamma-2}\tilde{Z}(s)^\top\left(Du(s,X(s))\sigma(s,X(s))-Du(s,\hat{X}(s))\sigma(s,\hat{X}(s))\right)\dd{W}(s)\\
		+&c\int_{t_1\wedge\tau}^{t_2\wedge\tau}\abs{Z(s)}^{2\gamma-2}\norm{Du(s,X(s))\sigma(s,X(s))-Du(s,\hat{X}(s))\sigma(s,\hat{X}(s))}_{HS}^2\dd{s}\\
		=&I_1+I_2+I_3+I_4
	\end{align*}
	where $c>0$ is a constant depending only on $\gamma,d,p,q,\kappa,K,T_0,\norm{b}_{L^q_p(T_0)},g$ and $R$. The idea is to apply the stochastic Gronwall Lemma \ref{Gronwall}. To get rid of the badly behaving terms $I_2$ and $I_4$, one can use a suitable multiplier of the form $e^{-A(t)}$ - like in \cite{Fedrizzi2} - where $A$ is an adapted, continuous process. Here, we choose
	\begin{align*}
		A(t)&:=c\int_{S\wedge\tau}^{t\wedge\tau}\abs{V(s,X_s)}\frac{\norm{Du(s,X(s))-Du(s,\hat{X}(s))}_{HS}}{\abs{Z(s)}}\1_{Z(s)\neq0}\dd{s}\\
		&+c\int_{S\wedge\tau}^{t\wedge\tau}\frac{\norm{Du(s,X(s))\sigma(s,X(s))-Du(s,\hat{X}(s))\sigma(s,\hat{X}(s))}_{HS}^2}{\abs{Z(s)}^2}\1_{Z(s)\neq0}\dd{s}
	\end{align*}
	for $S\wedge\tau\leq t\leq T\wedge\tau$.
	To show that $A$ is indeed well defined, it suffices to show the existence of a constant $\hat{C}=\hat{C}(\gamma,d,p,q,\kappa,K,T_0,\norm{b}_{L^q_p(T_0)},\norm{\nabla\sigma}_{L^q_p(T_0)},g,R)\geq0$ such that
	$$\E \exp\left(\frac{1}{2}A(T\wedge\tau)\right)\leq \hat{C}.$$
	Since $u$ belongs coordinatewise to $ H^q_{2,p}(T_0)$ and by conditions \eqref{sigmac} and \eqref{dsigmac}, it holds
	$$(Du\cdot\sigma)^{i,j}\in L^q\left(T_0;W^{1,p}\left(\R^d\right)\right), \ i,j=1,\dots,d.$$
	Additionally, $C^\infty_c\left(\R^{d+1}\right)$ is dense in $L^q\left(T_0;W^{1,p}\left(\R^d\right)\right)$. Hence, by Young's inequality, Lemmas \ref{exp2} and \ref{Kr4}, it suffices to show for all $\tilde{R}>0$ the existence of a constant $C_{\tilde{R}}=C_{\tilde{R}}(d,p,q,\kappa,T_0,\norm{b}_{L^q_p(T_0)},g,R)$ such that
	$$\E\exp\left(\int_{S\wedge\tau}^{T\wedge\tau}\frac{\abs{f(s,X(s))- f(s,\hat{X}(s))}^2}{\abs{Z(s)}^2}\1_{Z(s)\neq0}\dd{s}\right)\leq C_{\tilde{R}}$$ 
	for all $f\in C^\infty\left(\R^{d+1}\right)$ with $\norm{f}_{L^q\left(T_0;W^{1,p}\left(\R^d\right)\right)}\leq\tilde{R}$. By Lemmas \ref{Hardy-Littlewood} and \ref{Kr4}, one obtains
	\begin{align*}
		&\E\exp\left(\int_{S\wedge\tau}^{T\wedge\tau}\frac{\abs{f(s,X(s))- f(s,\hat{X}(s))}^2}{\abs{Z(s)}^2}\1_{Z(s)\neq0}\dd{s}\right)\\
		\leq&\E\exp\left(C_d^2\int_{S\wedge\tau}^{T\wedge\tau}\left(\mathcal{M}\abs{\nabla f}(X(s))+\mathcal{M}\abs{\nabla f}(\hat{X}(s))\right)^2\dd{s}\right)\\
		\leq&C_{\tilde{R}}
	\end{align*}
	where $C_{\tilde{R}}=C_{\tilde{R}}(d,p,q,\kappa,T_0,\norm{b}_{L^q_p(T_0)},g,R)$. By the It\=o formula, it holds
	$$e^{-A(t)}\abs{Z(t\wedge\tau)}^{2\gamma}\leq\abs{Z(S\wedge\tau)}^{2\gamma}+ c\int_{S\wedge\tau}^{t\wedge\tau}e^{-A(s)}\norm{Z_s}_\infty^{2\gamma}\dd{s}+\text{local martingale}.$$
	Applying the stochastic Gronwall Lemma \ref{Gronwall} gives
	$$\E\left[\sup_{S\wedge\tau\leq t\leq T\wedge\tau}e^{-\frac{1}{2}A(t)}\abs{Z(t)}^\gamma\right]\leq \tilde{C}\E\norm{Z_{S\wedge\tau}}_\infty^\gamma\leq\tilde{C}C_1\norm{x-\hat{x}}_\infty^\gamma$$
	for a constant $\tilde{C}=\tilde{C}(\gamma,d,p,q,\kappa,K,T_0,\norm{b}_{L^q_p(T_0)},\norm{\nabla\sigma}_{L^q_p(T_0)},g,R)$. Due to the estimates from above, the Cauchy-Schwarz inequality and by redefining $\gamma:=2\gamma$, one finally obtains
	\begin{align*}
		&\E\left[\sup_{S\wedge\tau\leq t\leq T\wedge\tau}\abs{Z(t)}^\gamma\right]\\
		\leq&\left(\E e^{\frac{1}{2}A(T\wedge\tau)}\right)^\frac{1}{2}\left[\E\left(\sup_{S\wedge\tau\leq t\leq T\wedge\tau}e^{-\frac{1}{2}A(t)}\abs{Z(t)}^{2\gamma}\right)\right]^{\frac{1}{2}}\\
		\leq&C_2\norm{x-\hat{x}}_\infty^\gamma
	\end{align*}
	for some constant $C_2=C_2(\gamma,d,p,q,\kappa,K,T_0,\norm{b}_{L^q_p(T_0)},\norm{\nabla\sigma}_{L^q_p(T_0)},g,R)$.
\end{proof}
\begin{Remark}
	The application of the stochastic Gronwall Lemma \ref{Gronwall} is crucial in the proof above, since one has to deal with the supremum norm of path segments. Another standard ansatz might be to apply Doob's or Burkholder's inequality. Unfortunately, it does not work due to the bad regularity of the quadratic variation term of the martingale part. Thus, the inequalities used in \cite{Zhang2011} or \cite{Fedrizzi2} are not suitable.
\end{Remark}
\begin{proof}[Proof of Theorem \ref{thm}]
	This theorem is a consequence of Theorems \ref{Existence}, \ref{Uniqueness}, Lemma \ref{Hoelder} and the Yamada-Watanabe Theorem (cf. \cite{yamada1971}).
\end{proof}
\begin{proof}[Proof of Theorem \ref{locthm}]
	Firstly, assume that $V$ is bounded and conditions \eqref{driftc}, \eqref{sigmac} and \eqref{dsigmac} are fulfilled. Let $X,\hat{X}\in\mathcal{S}^\tau(x)$ for a stopping time $\tau$. Sets of the type
	$$\mathcal{K}_n:=\left\{x\in\Co:\sup\limits_{t\in[-r,0]}\abs{x(t)}+\sup\limits_{-r\leq s<t\leq 0}\frac{\abs{x(t)-x(s)}}{\abs{t-s}^{1/4}}\leq n\right\}$$
	with $n\in\N$ are compact in $\Co$ and by Lemma \ref{Hoelder}, it holds
	$$\lim\limits_{n\to\infty}\tau^n=\tau$$
	for 
	$$\tau^n:=\inf\left\{t\leq\tau:X_t\notin\mathcal{K}_n\text{ or }\hat{X_t}\notin\mathcal{K}_n\right\}\wedge\tau\wedge n.$$
	By assumption, for each $n\in\N$, there exists a $C_{\mathcal{K}_n,n}>0$ such that $V$ is $C_{\mathcal{K}_n,n}$-Lipschitz continuous in space on $\mathcal{K}_n$. So, there exists a measurable, bounded, in space $C_{\mathcal{K}_n,n}$-Lipschitz continuous extension $V^n:\R_{\geq0}\times\Co\to\R^d$ of $V_{\big|[0,n]\times \mathcal{K}_n}$, i.e.
	$$V^n(t,x)=V(t,x) \ \forall x\in\mathcal{K}_n, \ 0\leq t\leq n.$$
	By Theorem \ref{thm}, for each $n\in\N$, there exists a global, unique strong solution $X^n$ for equation $\eqref{deq}$ with coefficients $V^n$, $b$ and $\sigma$. Therefore, one has
	$$X(t)=\hat{X}(t)=X^n(t), \ 0\leq t\leq\tau^n,$$
	which provides the pathwise uniqueness.\\
	For the general case, let again $X,\hat{X}\in\mathcal{S}^\tau$ for some stopping time $\tau$. Since $V$ is bounded on compact sets, one has
	$$\lim\limits_{n\to\infty}\tau_n=\tau$$
	for
	$$\tau_n:=\inf\left\{t\leq\tau:\abs{V(t,X_t)}>n,\abs{X(t)}>n,\abs{V(t,\hat{X}_t)}>n\text{ or }\abs{\hat{X}(t)}>n\right\}\wedge\tau\wedge n.$$
	Define for each $n\in\N$
	$$b^n(t,x):=\1_{t,\abs{x}\leq n}b(t,x), \ (t,x)\in\R_{\geq0}\times\R^d$$
	and
	$$\sigma^n(t,x):=\sigma(t,\phi^n(x)), \ (t,x)\in\R_{\geq0}\times\R^d$$
	where $\phi^n:\R^d\to B_{n+1}$ is a $C^1$-diffeomorphism defined by
	$$\phi(x):=
		\begin{cases}
			x, & \abs{x}\leq n,\\
			\rho^n(\abs{x})\frac{x}{\abs{x}}, & \abs{x}>n
		\end{cases}$$
	with
	\begin{align*}
		\rho^n(t)&:=n+1-\left(\frac{1}{\alpha_n}t-\frac{1}{\alpha_n}n+1\right)^{-\alpha_n}, \ t>n,\\
		\alpha_n&:=\frac{p_{n+1}-d}{2}.
	\end{align*}
	Furthermore, redefine
	$$V^n(t,x):=(-n)\vee V(t,x)\wedge n.$$
	By the previous discussion, one has for each $n\in\N$ a global, unique strong solution $X^n$ of equation \eqref{deq} with coefficients $V^n$, $b^n$ and $\sigma^n$ and it holds
	$$X(t)=X^n(t)=\hat{X}(t), \ 0\leq t\leq \tau_n$$
	which provides the pathwise uniqueness and the stated maximal solution until the explosion time $\zeta$. 
\end{proof}
\begin{proof}[Proof of Theorem \ref{constthm}]
	The Lipschitz condition \eqref{lip} is not necessary for any result in section 2. Accordingly, one can follow the proof of Theorem \ref{Uniqueness} in exactly the same way. Using the same notation, the term
	$$\int_{t_1\wedge\tau}^{t_2\wedge\tau}2\gamma\abs{\tilde{Z}(s)}^{2\gamma-2}\tilde{Z}(s)^\top\left(Du(s,X(s))V(s,X_s)-Du(s,\hat{X}(s))V(s,\hat{X}_s)\right)\dd{s}$$
	was split into
	\begin{align*}
		&\int_{t_1\wedge\tau}^{t_2\wedge\tau}2\gamma\abs{\tilde{Z}(s)}^{2\gamma-2}\tilde{Z}(s)^\top\left(Du(s,X(s))V(s,X_s)-Du(s,\hat{X}(s))V(s,\hat{X}_s)\right)\dd{s}\\
		=&\int_{t_1\wedge\tau}^{t_2\wedge\tau}2\gamma\abs{\tilde{Z}(s)}^{2\gamma-2}\tilde{Z}(s)^\top\left(Du(s,\hat{X}(s))V(s,X_s)-Du(s,\hat{X}(s))V(s,\hat{X}_s)\right)\dd{s}\\
		&+\int_{t_1\wedge\tau}^{t_2\wedge\tau}2\gamma\abs{\tilde{Z}(s)}^{2\gamma-2}\tilde{Z}(s)^\top\left(Du(s,X(s))V(s,X_s)-Du(s,\hat{X}(s))V(s,X_s)\right)\dd{s}
	\end{align*}
	where $S\leq t_1\leq t_2\leq T$. The Lipschitz condition \eqref{lip} was only used to estimate the first summand by using 
	$$\abs{V(t,X_t)-V(t,\hat{X}_t)}\leq K\norm{X_t-\hat{X}_t}_\infty, \ t\geq0.$$
	If one can show that the same inequality still holds for two solutions $X\in\mathcal{S}^\tau(x)$ and $\hat{X}\in\mathcal{S}^\tau(x)$ with the same initial value $x\in\Co$, the claimed pathwise uniqueness will follow. Now, one has
	$$X(t)-\hat{X}(t)=\int_{0}^{t}\left(b(s,X(s))-b(s,\hat{X}(s))\right)\dd{s}$$
	since $\sigma$ is assumed to be space-independent. Together with Lemma \ref{Kr4}, it follows that a.s.
	$$\left([-r,t]\ni s\mapsto X(s)-\hat{X}(s)\right)\in\mathcal{H}_t.$$
	Consequently, by rewriting
	$$V(t,X_t)-V(\hat{X}_t)=V(t,X_t)-V(t,X_t+\hat{X}_t-X_t),$$
	one can apply the assumption and ends up with the desired estimate. The global existence is given by Theorem \ref{Existence}. 
\end{proof}
\appendix
\section{Appendices}
\begin{Theorem}
	\label{PDE}
	Assume conditions \eqref{driftc} and \eqref{sigmac}. Then for any $T>0$ and $f\in L^q_p(T)$, there exists a unique solution $u\in H^q_{2,p}(T)$ of the following PDE
	\begin{align*}
	\partial_tu(t,x)+L_tu(t,x)+f(t,x)&=0,\\
	u(T,x)&=0
	\end{align*}
	where
	$$L_tu(t,x):=\frac{1}{2}\sum_{i,j,k=1}^{d}\sigma^{i,k}(t,x)\sigma^{j,k}(t,x)\partial_i\partial_ju(t,x)+b(t,x)\cdot\nabla u(t,x)$$
	with the bound
	$$\norm{u}_{H^q_{2,p}(S,T)}\leq C\norm{f}_{L^q_p(S,T)}$$
	for any $S\in[0,T]$ and some constant $C=C(T,\kappa,p,q,\norm{b}_{L^q_p(T)})>0$.
\end{Theorem}
\begin{proof}
	See \cite{Zhang2011}.
\end{proof}
\begin{Theorem}
	\label{embedding}
	Let $p,q\in(1,\infty)$, $T>0$ and $u\in H^q_{2,p}(T)$.
	\begin{enumerate}
		\item If $\frac{d}{p}+\frac{2}{q}<2$, then $u$ is a bounded H\"older continuous function on $[0,T]\times\R^d$ and for any $0<\varepsilon$, $\delta\leq1$ satisfying
		$$\varepsilon+\frac{d}{p}+\frac{2}{q}<2, \ \ 2\delta+\frac{d}{p}+\frac{2}{q}<2,$$
		there exists a constant $N=N(p,q,\varepsilon,\delta)$ such that
		\begin{align*}
			\abs{u(t,x)-u(s,x)}&\leq N\abs{t-s}^\delta\norm{u}^{1-\frac{1}{q}-\delta}_{\mathds{H}^q_{2,p}(T)}\norm{\partial_tu}^{\frac{1}{q}+\delta}_{L^q_p(T)},\\
			\abs{u(t,x)}+\frac{\abs{u(t,x)-u(t,y)}}{\abs{x-y}^\varepsilon}&\leq NT^{-\frac{1}{q}}\left(\norm{u}_{\mathds{H}^q_{2,p}(T)}+T\norm{\partial_tu}_{L^q_p(T)}\right)
		\end{align*}
		for all $s,t\in[0,T]$ and $x,y\in\R^d,x\neq y$.
		\item If $\frac{d}{p}+\frac{2}{q}<1$, then $\nabla u$ is a bounded H\"older continuous function on $[0,T]\times\R^d$ and for any $\varepsilon\in(0,1)$ satisfying
		$$\varepsilon+\frac{d}{p}+\frac{2}{q}<1,$$
		there exists a constant $N=N(p,q,\varepsilon)$ such that
		\begin{align*}
		\abs{\nabla u(t,x)-\nabla u(s,x)}&\leq N\abs{t-s}^\delta\norm{ u}^{1-\frac{1}{q}-\frac{\varepsilon}{2}}_{\mathds{H}^q_{2,p}(T)}\norm{\partial_tu}^{\frac{1}{q}+\frac{\varepsilon}{2}}_{L^q_p(T)},\\
		\abs{\nabla u(t,x)}+\frac{\abs{\nabla u(t,x)-\nabla u(t,y)}}{\abs{x-y}^\varepsilon}&\leq NT^{-\frac{1}{q}}\left(\norm{u}_{\mathds{H}^q_{2,p}(T)}+T\norm{\partial_tu}_{L^q_p(T)}\right)
		\end{align*}
		for all $s,t\in[0,T]$ and $x,y\in\R^d,x\neq y$.
	\end{enumerate}
\end{Theorem}
\begin{proof}
	See \cite[p. 22, 23, 36]{Fedrizzi}.
\end{proof}
In the next lemma we identify every $u\in H^q_{2,p}$ with its regular version.
\begin{Lemma}[It\=o formula for $H^q_{2,p}$-functions]
	\label{Ito}
	Let $T>0$, $p>1$ and $q>1$ satisfying \eqref{pq}. Let $X:\Omega\times[0,T]\to\R^d$ be a semimartingale on some filtrated probability space $\left(\Omega,\mathcal{F},\Prob,\left(\mathcal{F}_t\right)_{t\geq0}\right)$ of the form
	$$\dd{X}(t)=b(t)\dd{t}+\sigma(t)\dd{W}(t)$$
	where $(W_t)_{t\geq0}$ is a standard $d$-dimensional Brownian motion, $b:\Omega\times[0,T]\to\R^d$ and $\sigma:\Omega\times[0,T]\to\R^{d\times d}$ are progressively measurable with
	$$\Prob\left(\norm{b}_{L^1[0,T]}+\norm{a^{i,j}}_{L^\delta[0,T]}<\infty\right)=1, \ i,j=1,\dots,d$$
	for some $1<\delta\leq\infty$ where $a:=\sigma\sigma^\top$. Furthermore, assume that there exists a constant $C>0$ with
	$$\E\int_{0}^{T}f(t,X(t))\dd{t}\leq C\norm{f}_{L^{q/\delta*}_{p/\delta^*}(T)}$$
	for all $f\in L^{q/\delta^*}_{p/\delta^*}(T)$ where $\delta^*$ denotes the conjugate exponent of $\delta$.
	Then for any $u\in H^q_{2,p}(T)$, the It\=o formula holds, i.e.
	\begin{align*}
		u(t,X(t))-u(0,X(0))=&\int_{0}^{t}\partial_tu(s,X(s))\dd{s}+\int_{0}^{t}\nabla u(s,X(s))^\top b(s)\dd{s}\\
		&+\int_{0}^{t}\nabla u(s,X(s))^\top\sigma(s)\dd{W}(s)\\
		&+\frac{1}{2}\sum_{i,j=1}^{d}\int_{0}^{t}\partial_i\partial_j u(s,X(s))a^{i,j}(s)\dd{s}.
	\end{align*}
\end{Lemma}
\begin{proof}
	One can assume without loss of generality
	$$\norm{b}_{L^1(\Omega\times[0,1])}+\norm{a^{i,j}}_{L^\delta(\Omega\times[0,T])}<\infty, \ i,j=1,\dots,d$$
	by using the standard localization argument via stopping times. Next, choose a sequence $(u_n)_{n\in\N}\subset C_c^\infty\left(\R^{d+1}\right)$ such that
	$$\lim\limits_{n\to\infty}\norm{u-u_n}_{H^q_{2,p}(T)}=0.$$
	By the embedding Theorem \ref{embedding}, it holds
	$$\lim\limits_{n\to\infty}\left(\norm{u_n-u}_{L^\infty(\R^{d+1})}+\norm{\nabla u-\nabla u_n}_{L^\infty(\R^{d+1},\R^d)}\right)=0.$$
	The It\=o formula gives for each $n\in\N$ and $t\in[0,T]$
	\begin{align*}
		u_n(t,X(t))-u_n(0,X(0))=&\int_{0}^{t}\partial_tu_n(s,X(s))\dd{s}+\int_{0}^{t}\nabla u_n(s,X(s))^\top b(s)\dd{s}\\
		&+\int_{0}^{t}\nabla u_n(s,X(s))^\top\sigma(s)\dd{W}(s)\\
		&+\frac{1}{2}\sum_{i,j=1}^{d}\int_{0}^{t}\partial_i\partial_j u_n(s,X(s))a^{i,j}(s)\dd{s}.
	\end{align*}
	The left-hand side converges to $u(t,X(t))-u(0,X(0))$ by the choice of $u_n$. Furthermore, for $\delta<\infty$, one has the following four estimates
	\begin{align*}
		&\E\int_{0}^{t}\abs{\partial_tu(s,X(s))-\partial_tu_n(s,X(s))}\dd{s}\\
		\leq&T^{\frac{1}{\delta}}\left(\E\int_{0}^{T}\abs{\partial_tu(t,X(t))-\partial_tu_n(t,X(t))}^{\delta^*}\dd{t}\right)^{\frac{1}{\delta^*}}\\
		\leq&CT^{\frac{1}{\delta}}\norm{\abs{\partial_tu-\partial_tu_n}^{\delta^*}}_{L^{q/\delta^*}_{p/\delta^*}(T)}^\frac{1}{\delta^*}\\
		=&CT^{\frac{1}{\delta}}\norm{\partial_tu-\partial_tu_n}_{L^q_p(T)},		
	\end{align*}
	$$\E\int_{0}^{t}\abs{\nabla u(s)-\nabla u_n(s)}\abs{b(s)}\dd{s}\leq\norm{\nabla u-\nabla u_n}_{L^\infty(\R^{d+1},\R^d)}\norm{b}_{L^1(\Omega\times[0,T])},$$
	\begin{align*}
		&\E\abs{\int_{0}^{t}(\nabla u(s,X(s))-\nabla u_n(s,X(s)))^\top\sigma(s)\dd{W}(s)}^2\\
		\leq&\E\int_{0}^{T}\abs{\left(\nabla u(t,X(t))-\nabla u_n(t,X(t))\right)\sigma(t)}^2\dd{t}\\
		\leq&\norm{\nabla u-\nabla u_n}_{L^\infty(\R^{d+1},\R^d)}^2\cdot\E\int_{0}^{T}\norm{\sigma(t)}_{HS}^2\dd{t},
	\end{align*}
	and
	\begin{align*}
		&\E\sum_{i,j=1}^{d}\int_{0}^{t}\abs{\partial_i\partial_ju(s,X(s))-\partial_i\partial_ju_n(s,X(s))}\abs{a^{i,j}(s)}\dd{s}\\
		\leq&\sum_{i,j=1}^{d}\left(\E\int_{0}^{T}\abs{\partial_i\partial_ju(t,X(t))-\partial_i\partial_ju_n(t,X(t))}^{\delta^*}\dd{t}\right)^{\frac{1}{\delta^*}}\norm{a^{i,j}}_{L^\delta(\Omega\times[0,T])}\\
		\leq&\sum_{i,j=1}^{d}C\norm{\abs{\partial_i\partial_ju-\partial_i\partial_ju_n}^{\delta^*}}_{L^{q/\delta^*}_{p/\delta^*}(T)}^{\frac{1}{\delta^*}}\norm{a^{i,j}}_{L^\delta(\Omega\times[0,T])}\\
		=&\sum_{i,j=1}^{d}C\norm{\partial_i\partial_ju-\partial_i\partial_ju_n}_{L^q_p(T)}\norm{a^{i,j}}_{L^\delta(\Omega\times[0,T])}.
	\end{align*}
	For $\delta=\infty$, the estimates are basically the same. All these terms above converge to zero by the choice of $u_n$, which provides the desired convergence of the right-hand side.
\end{proof}
Let $\phi$ be a locally integrable function on $\R^d$. The Hardy-Littlewood maximal function is definded by
$$\mathcal{M}\phi(x):=\sup_{0<r<\infty}\frac{1}{\abs{B_r}}\int_{B_r}\phi(x+y)\dd{y}$$
where $B_r$ is the Euclidean ball of radius $r$. The following result is cited from \cite{Zhang2011}.
\begin{Lemma} \ 
	\label{Hardy-Littlewood}
	\begin{enumerate}
		\item	There exists a constant $C_d>0$ such that for all $\phi\in C^{\infty}\left(\R^d\right)$ and $x,y\in\R^d$,
		$$\abs{\phi(x)-\phi(y)}\leq C_d\abs{x-y}\left(\mathcal{M}\abs{\nabla\phi}(x)+\mathcal{M}\abs{\nabla\phi}(y)\right).$$
		\item For any $p>1$, there exists a constant $C_{d,p}$ such that for all $\phi\in L^p\left(\R^d\right)$,
		$$\norm{\mathcal{M}\phi}_{L^p}\leq C_{d,p}\norm{\phi}_{L^p.}$$
	\end{enumerate}
\end{Lemma}
\begin{Lemma}
	\label{Gronwall}
	Let $Z$ be an adapted non-negative stochastic process with continuous paths defined on $[0,\infty)$ which satisfies the inequality
	$$Z(t)\leq K\int_{0}^{t}\sup\limits_{0\leq r\leq s}Z(r)\dd{s}+M(t)+C,$$
	where $C\geq0$, $K\geq0$ and $M$ is a continuous local martingale with $M(0)=0$. Then for each $0<p<1$, there exist universal finite constants $c_1(p)$, $c_2(p)$ (not depending on $K$, $C$, $T$ and $M$) such that
	$$\E\left[\sup\limits_{0\leq t\leq T}Z(t)^p\right]\leq C^p c_2(p)e^{c_1(p)KT}\text{ for every }T\geq0.$$
\end{Lemma}
\begin{proof}
	See \cite{Renesse,Scheutzow}.
\end{proof}
\begin{Lemma}[Modified Khas'minskii lemma]
	\label{Khas}
	Let $\beta:\Omega\times[0,T]\to\R_{\geq0}$ be a non-negative, measurable, adapted process with respect to some filtrated probability space $\left(\Omega,\mathcal{F},\Prob,\left(\mathcal{F}_t\right)_{0\leq t\leq T}\right)$ and $T>0$. Assume there exists some $0\leq\alpha<1$ for all $0\leq s\leq t\leq T$ such that
	$$\E\left[\int_{s}^{t}\beta(r)\dd{r}\bigg|\mathcal{F}_s\right]\leq\alpha.$$
	Then one has
	$$\E\exp\left(\int_{0}^{T}\beta(r)\dd{r}\right)\leq\frac{1}{1-\alpha}.$$
\end{Lemma}
\begin{proof} This proof mainly follows the technique used in \cite{Fedrizzi}. 
	\begin{align*}
	& \  \  \ \E\exp\left(\int_{0}^{T}\beta(r)\dd{r}\right)\\
	&=\sum_{n=0}^{\infty}\frac{1}{n!}\E\int_{0}^{T}\cdots\int_{0}^{T}\beta(r_1)\cdots\beta(r_n)\dd{r_1}\cdots\dd{r_n}\\
	&=\sum_{n=0}^{\infty}\E\int_{0\leq r_1\leq r_2\leq\cdots\leq r_n\leq T}\beta(r_1)\cdots\beta(r_n)\dd{r}\\
	&=\sum_{n=0}^{\infty}\E\left[\int_{0\leq r_1\leq r_2\leq\cdots\leq r_{n-1}\leq T}\beta(r_1)\cdots\beta(r_{n-1})\int_{r_{n-1}}^{T}\beta(r_n)\dd{r_n}\dd{r_1}\cdots\dd{r_{n-1}}\right]\\
	&=\sum_{n=0}^{\infty}\E\left[\int_{0}^{T}\cdots\int_{0}^{r_2}\beta(r_1)\cdots\beta(r_{n-1})\E\left(\int_{r_{n-1}}^{T}\beta(r_n)\dd{r_n}\bigg|\mathcal{F}_{r_{n-1}}\right)\dd{r_1}\cdots\dd{r_{n-1}}\right]\\
	&\leq\sum_{n=0}^{\infty}\alpha\E\int_{0}^{T}\cdots\int_{0}^{r_2}\beta(r_1)\cdots\beta(r_{n-1})\dd{r_1}\cdots\dd{r_{n-1}}\\
	&\overset{(\star)}{\leq}\sum_{n=0}^{\infty}\alpha^n\\
	&=\frac{1}{1-\alpha}
	\end{align*}
	where $(\star)$ was obtained by iteration.
\end{proof}

\end{document}